\documentclass{article}

\usepackage{arxiv}

\usepackage[utf8]{inputenc} 
\usepackage[T1]{fontenc}    
\usepackage{hyperref}       
\usepackage{url}            
\usepackage{booktabs}       
\usepackage{amsfonts}       
\usepackage{nicefrac}       
\usepackage{microtype}      
\usepackage{graphicx}
\usepackage{amsmath}
\usepackage{amsthm}
\newtheorem{thm}{Theorem}[section]
\newtheorem{lem}[thm]{Lemma}
\newtheorem{prop}[thm]{Proposition}
\newtheorem{cor}[thm]{Corollary}
\newtheorem{exmp}{Example}[section]
\newtheorem{rem}{Remark}

\title{Finite, fiber-preserving group actions on elliptic 3-manifolds}

\author{
  Benjamin Peet\\
  Department of Mathematics\\
  St. Martin's University\\
  Lacey, WA 98503 \\
  \texttt{bpeet@stmartin.edu} \\
}

\begin{document}
\maketitle

\begin{abstract}
In two previous papers the author presented a general construction of finite, fiber- and orientation-preserving group actions on orientable Seifert manifolds. In this paper we restrict our attention to elliptic 3-manifolds. A proof is given that orientation-reversing and fiber-preserving diffeomorphisms of Seifert manifolds do not exist for nonzero Euler class, in particular elliptic 3-manifolds. Each type of elliptic 3-manifold is then considered and the possible group actions that fit the given construction. This is shown to be all but a few cases that have been considered elsewhere. Finally, a presentation for the quotient space under such an action is constructed and a specific example is generated.
\end{abstract}

\keywords{geometry; topology; $3$-manifolds; finite group actions; Seifert fiberings; elliptic}

\section{Introduction}

\subsection{Discussion of results}

In previous papers \cite{peet2018} and \cite{peet2018finite} we considered orientable Seifert manifolds and the possible finite groups that can act fiber- and orientation-preservingly.

The main results in those papers established that firstly:

\begin{thm}
Let $M$ be a closed, compact, and orientable Seifert $3$-manifold that fibers over an orientable base space. Let $\varphi:G\rightarrow Diff_{+}^{fp}(M)$ be a finite group action on $M$ such that the obstruction class can expressed as $$b=\sum_{i=1}^{m}(b_{i}\cdot\#Orb_{\varphi}(\alpha_{i}))$$ for a collection of fibers $\{\alpha_{1},\ldots,\alpha_{m}\}$ and integers $\{b_{1},\ldots,b_{m}\}$. Then $\varphi$ is an extended product action.
\end{thm}

Where an extended product action is intuitively a product action on an orientable surface with boundary cross $S^{1}$ extended across Dehn fillings of the boundary tori.

Secondly, it was shown that:

\begin{cor}
 Suppose that $\varphi:G\rightarrow Diff(M)$ is a finite group action on an orientable Seifert manifold with a non-orientable base space. Then provided that the unique lifted group action $\tilde{\varphi}:G\rightarrow Diff(\tilde{M})$ satisfies the obstruction condition, $G$ is isomorphic to a subgroup of $\mathbb{Z}_{2}\times H$ where $H$ is a finite group that acts orientation-preservingly on the orientable base space of $\tilde{M}$.
\end{cor}

These two results will allow us to consider the elliptic 3-manifolds in particular and present the possible finite, fiber- and orientation-preserving groups that can act on them.

We then present a proof that all finite, fiber-preserving actions on Seifert manifolds with non-zero Euler class must be orientation-preserving and in particular apply this to elliptic manifolds.

Finally, we consider the quotient orbifolds that will arise under the given actions and present a thorough example of one such action.

\subsection{Preliminary definitions}

Throughout we consider $M$ to be an oriented smooth manifold of dimension $3$ and without boundary. $G$ is always assumed to be a finite group. We denote $Diff(M)$ as the group of self-diffeomorphisms of $M$, and then define a $G$-action on $M$ as an injection $\varphi:G\rightarrow Diff(M)$. We use the notation $Diff_{+}(M)$ for the group of orientation-preserving self-diffeomorphisms of $M$. 

$M$ is further assumed to be an orientable Seifert-fibered manifold. That is, $M$ can be decomposed into disjoint fibers where each fiber is a simple closed curve and each fiber has a fibered neighborhood which can be mapped under a fiber-preserving map onto a solid fibered torus.

A Seifert bundle is a Seifert manifold $M$ along with a continuous map $p:M\rightarrow B$ where $p$ identifies each fiber to a point. For clarity, we denote the underlying space of $B$ as $B_{U}$. 

We use the normalized notation $(g,\epsilon|(q_{1},p_{1}),\ldots,(q_{n},p_{n}),(1,b))$ to indicate an orientable Seifert manifold with normalized Seifert invariants $(q_{1},p_{1}),\ldots,(q_{n},p_{n})$, obstruction class $b$, and $\epsilon=o_{1}$ if the base space is orientable and $\epsilon=n_{2}$ is the base space is not orientable.

The Euler class of a Seifert manifold with normalized Seifert manifold is given by $e=-(b+\sum_{i=1}^{n}\frac{p_{i}}{q_{i}})$ and an elliptic Seifert manifold is such that $e=0$ and the base orbifold has positive Euler characteristic.

We say a $G$-action is fiber-preserving if for any fiber $\gamma$ and any $g\in G$, $\varphi(g)(\gamma)$ is some fiber of $M$. We use the notation $Diff^{fp}(M)$ for the group of fiber-preserving self-diffeomorphisms of $M$ (given some Seifert fibration). Given a fiber-preserving $G$-action, there is an induced action $\varphi_{B_{U}}:G_{B_{U}}\rightarrow Diff(B_{U})$ on the underlying space $B_{U}$ of the base space $B$. 

Given a finite action $\varphi:G\rightarrow Diff^{fp}(M)$, we define the orbit number of a fiber $\gamma$ under the action to be $\#Orb_{\varphi}(\gamma)=\#\{\alpha|\varphi(g)(\gamma)=\alpha\textrm{ for some }g\in G\}$. 

If we have a manifold $M$, then a product structure on $M$ is a diffeomorphism $k:A\times B\rightarrow M$ for some manifolds $A$ and $B$. \cite{lee2003smooth} If a Seifert-fibered manifold $M$ has a product structure $k:S^{1}\times F\rightarrow M$ for some surface $F$ and $k(S^{1}\times\{x\})$ are the fibers of $M$ for each $x\in F$, then we say that $k:S^{1}\times F\rightarrow M$ is a fibering product structure of $M$. 

Given that the first homology group (equivalently the first fundamental group) of a torus is $\mathbb{Z}\times\mathbb{Z}$ generated by two elements represented by any two nontrivial loops that cross at a single point, we can use the meridian-longitude framing from a product structure as representatives of two generators. If we have a diffeomorphism $f:T_{1}\rightarrow T_{2}$ and product structures $k_{i}:S^{1}\times S^{1}\rightarrow T_{i}$, then we can express the induced map on the first homology groups by a matrix that uses bases for $H_{1}(T_{i})$ derived from the meridian-longitude framings that arise from $k_{i}:S^{1}\times S^{1}\rightarrow T_{i}$. We denote this matrix as $\left[\begin{array}{cc}
a_{11} & a_{12}\\
a_{21} & a_{22}
\end{array}\right]_{k_{2}}^{k_{1}}$ .

We say that a $G$-action $\varphi:G\rightarrow Diff(A\times B)$ is a product action if for each $g\in G$, the diffeomorphism $\varphi(g):A\times B\rightarrow A\times B$ can be expressed as $(\varphi_{1}(g),\varphi_{2}(g))$ where $\varphi_{1}(g):A\rightarrow A$ and $\varphi_{2}(g):B\rightarrow B$. Here $\varphi_{1}:G\rightarrow Diff(A)$ and $\varphi_{2}:G\rightarrow Diff(B)$ are not necessarily injections. 

Given an action $\varphi:G\rightarrow Diff(M)$ and a product structure $k:A\times B\rightarrow M$, we say that $\varphi$ leaves the product structure $k:A\times B\rightarrow M$ invariant if $\psi(g)=k^{-1}\circ\varphi(g)\circ k$ defines a product action $\psi:G\rightarrow Diff(A\times B)$.

Suppose that we now have a fibering product structure $k:S^{1}\times F\rightarrow M$. We then say that each boundary torus is positively oriented if the fibers are given an arbitrary orientation and then each boundary component of $k(\{u\}\times F)$ is oriented by taking the normal vector to the surface according the orientation of the fibers.

We consider two particular types of 3-orbifold. We define the solid torus with exceptional core $V(k)$ to be a solid torus with an exceptional set of order $k$ running along the core loop of the solid torus. We define the Conway ball $B(k)$ to be a ball with exceptional set consisting of two arcs of order two joined by an arc of order $k$ according to Figure 1 below:

\begin{figure}[ht]
\centering
\includegraphics[height=3cm]{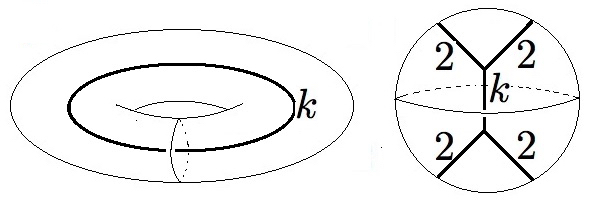}
\caption{A solid torus $V(k)$ and Conway ball $B(k)$ }
\end{figure}

\section{Preliminary results}

We begin with some preliminary results that we will use in the next section regarding orientation-reversing diffeomorphisms.

\begin{lem}Let $F$ be an orientable surface with boundary. Let the boundary be positively oriented according to some orientation of $F$ and $f:F\rightarrow F$ be a diffeomorphism. Then $f$ is orientation-preserving on $F$ if and only if $f$ is orientation-preserving between some pair of boundary components.
\end{lem}

\begin{proof}
In a regular neighborhood of two exchanged boundary components (they may be the same), the diffeomorphism is either a reflection or a rotation (given parameterizations of the annuli). If it is a reflection, the orientation on the boundary is reversed and and the orientation on $F$ is reversed. If it is a rotation, the orientation on the boundary is preserved and and the orientation on $F$ is preserved. 
\end{proof}

\begin{cor} Let $\hat{M}$ be an oriented trivially Seifert fibered 3-manifold with positively oriented boundary $\partial\hat{M}=T_{1}\cup\ldots\cup T_{n}$. Then a fiber-preserving diffeomorphism $f:\hat{M}\rightarrow\hat{M}$ is orientation-preserving if and only if $f$ is orientation-preserving between some pair of boundary tori.
\end{cor}

\begin{proof}
Firstly, there is a fibering product structure $k:S^{1}\times F\rightarrow\hat{M}$. Suppose that the diffeomorphism preserves the orientation of the fibers. Then the projected diffeomorphism on $F$ must be orientation-preserving. By Lemma 2.1, this is if and only if it is orientation-preserving between some pair of boundary components. As the diffeomorphism preserves the orientation of a fiber, this is equivalent to $f$ being orientation-preserving between some pair of boundary tori. 

If now we suppose that the diffeomorphism reverses the orientation of the fibers. Then the projected diffeomorphism on $F$ must be orientation-reversing. By Lemma 2.1, this is if and only if it is orientation-reversing between some pair of boundary components. As the diffeomorphism reverses the orientation of a fiber, this is equivalent to $f$ being orientation-preserving between some pair of boundary tori. 
\end{proof}

\section{Conditions for an orientation-reversing action}

We now use the previous section to establish some results about the conditions under which an orientation-reversing action is possible.

Firstly, a condition on the order of critical fibers:

\begin{prop} All finite, fiber-preserving actions on an orientable Seifert 3-manifold fibering over an orientable base space with at least one critical fiber of order greater than two are orientation-preserving.
\end{prop}

\begin{proof}
Suppose for contradiction that there exists a periodic, fiber-preserving and orientation-reversing diffeomorphism $f:M\rightarrow M$. 

We begin with normalized invariants for $M=(g,o_{1}|(q_{1},p_{1}),\ldots,(q_{n},p_{n}),(1,b))$.

We then take a regular fiber $\gamma$ with $\#Orb_{f}(\gamma)=l$ for some $l$. Then adjust the invariants to yield $M=(g,o_{1}|(q_{1},p_{1}),\ldots,(q_{n},p_{n}),(1,b_{1}),\ldots,(1,b_{l}))$ where each $(1,b_{i})$ refer to a fiber in $Orb_{f}(\gamma)$. Necessarily, $\sum_{i=1}^{l}b_{i}=b$.

We can then proceed as in \cite{peet2018} to yield a manifold $\hat{M}$ with fibering product structure $k_{\hat{M}}:S^{1}\times F\rightarrow\hat{M}$ and a collection of solid tori $X$ with product structure $k_{X}:S^{1}\times(D_{1}\cup\ldots\cup D_{n+l})\rightarrow X$. 

We can also now define a restricted map $\hat{f}\in Diff(\hat{M})$. Suppose that the filling of $T_{i}$ yields a critical fiber of order greater than 2. 

Suppose that $\hat{f}(T_{i})=T_{j}$. It could be that $i=j$.

According to the given product structures (with positively oriented restrictions on the boundary) we then have the following homological diagram:

$$\begin{array}{ccccc}
 &  & (d|_{\partial V_{i}})_{*}\\
 & H_{1}(T_{i}) & \leftarrow & H_{1}(\partial V_{i})\\
\hat{f}_{*} & \downarrow &  & \downarrow & (d|_{\partial V_{j}}^{-1}\circ\hat{f}\circ d|_{\partial V_{i}})_{*}\\
 & H_{1}(T_{j}) & \leftarrow & H_{1}(\partial V_{j})\\
 &  & (d|_{\partial V_{j}})_{*}
\end{array}$$

Now, $f:M\rightarrow M$ is orientation-reversing and extends into the solid tori $V_{i},V_{j}$, hence $(d|_{\partial V_{j}}^{-1}\circ\hat{f}\circ d|_{\partial V_{i}})_{*}=\pm\left[\begin{array}{cc}
1 & 0\\
a & 1
\end{array}\right]_{k_{\partial V_{j}}}^{k_{\partial V_{i}}}$ or $\pm\left[\begin{array}{cc}
-1 & 0\\
a & 1
\end{array}\right]_{k_{\partial V_{j}}}^{k_{\partial V_{i}}}$. By Corollary 2.2, we must have the second case.

Then according to the framings on $V_{i},V_{j}$, the fibrations are given by a $(-q_{i},y_{i})=(-q_{j},y_{j})$ curve where $q_{i}=q_{j}>2$. $\hat{f}$ must preserve the fibration hence:

$$\pm\left[\begin{array}{cc}
-1 & 0\\
a & 1
\end{array}\right]\left[\begin{array}{c}
-q_{i}\\
y_{i}
\end{array}\right]=\pm\left[\begin{array}{c}
q_{i}\\
-y_{i}
\end{array}\right] $$

But this implies that $-aq_{i}+y_{i}=-y_{i}$, and so $aq_{i}=2y_{i}$. This further implies that $q_{i}$ divides 2 which is a contradiction. 
\end{proof}

Secondly, we establish that if the Euler class of the manifold is non-zero, then there are no orientation-reversing actions:

\begin{prop} All finite, fiber-preserving actions on an orientable Seifert 3-manifold fibering over an orientable base space with nonzero Euler class are orientation-preserving.
\end{prop}

\begin{proof} Again suppose for contradiction that there exists a periodic, fiber-preserving and orientation-reversing diffeomorphism $f:M\rightarrow M$. We proceed as in the previous proposition to yield a manifold $\hat{M}$ with fibering product structure $k_{\hat{M}}:S^{1}\times F\rightarrow\hat{M}$, a collection of solid tori $X$ with product structure $k_{X}:S^{1}\times(D_{1}\cup\ldots\cup D_{n+l})\rightarrow X$, and a restricted diffeomorphism $\hat{f}:\hat{M}\rightarrow\hat{M}$. 

We now consider the first homology group of $\hat{M}$. We have the presentation:

$$H_{1}(\hat{M})=\left\langle \alpha_{1},\ldots,\alpha_{n+l},a_{1},b_{1},\ldots,a_{g},b_{g},t|\alpha_{1}\cdots\alpha_{n+l}=1,\textrm{all commute}\right\rangle$$ 

Where $t$ represents an oriented fiber and $\alpha_{1},\ldots,\alpha_{n+l}$ represent positively oriented loops $k_{T_{i}}(\{u\}\times S^{1})$ on each boundary torus.

So we must have:

$$\hat{f}_{*}(\alpha_{i})=\alpha_{j(i)}^{\pm1}t^{c_{i}}$$

For some integer $c_{i}$ and some permutation $j\in perm\{1,\ldots n+l\}$. 

Here the sign is the same for each $\alpha_{i}$. So then:

$$1=\hat{f}_{*}(\alpha_{1}\cdots\alpha_{l})=t^{\sum_{i=1}^{n+l}c_{i}}$$

Hence, $$\sum _{i=1}^{n+l}c_{i}=0$$

\underline{Case 1:} There are no critical fibers. That is, $n=0$. 

Hence the obstruction is nonzero. We then consider the diagram:

$$\begin{array}{ccccc}
 &  & (d|_{\partial V_{i}})_{*}\\
 & H_{1}(T_{i}) & \leftarrow & H_{1}(\partial V_{i})\\
(\hat{f}|_{T_{i}})_{*} & \downarrow &  & \downarrow & (d|_{\partial V_{j(i)}}^{-1}\circ\hat{f}|_{T_{i}}\circ d|_{\partial V_{i}})_{*}\\
 & H_{1}(T_{j(i)}) & \leftarrow & H_{1}(\partial V_{j(i)})\\
 &  & (d|_{\partial V_{j(i)}})_{*}
\end{array}$$

So now $(d|_{\partial V_{i}})_{*}=\left[\begin{array}{cc}
-1 & b_{i}\\
0 & 1
\end{array}\right]_{k_{T_{i}}}^{k_{\partial V_{i}}} and (d|_{\partial V_{j(i)}}^{-1}\circ\hat{f}|_{T_{i}}\circ d|_{\partial V_{i}})_{*}=\pm\left[\begin{array}{cc}
1 & 0\\
0 & -1
\end{array}\right]_{k_{\partial V_{j(i)}}}^{k_{\partial V_{i}}}$ . This is as the diffeomorphism extends, is fiber-preserving, and orientation-reversing as well as each $V_{i}$ being trivially fibered. Here we again use Corollary 2.2.

Hence: $(\hat{f}|_{T_{i}})_{*}=\pm\left[\begin{array}{cc}
1 & -(b_{i}+b_{j(i)})\\
0 & -1
\end{array}\right]_{k_{T_{j(i)}}}^{k_{T_{i}}}$ . So then from above, $c_{i}=\mp(b_{i}+b_{j})$. Hence, we have:

$$\sum_{i=1}^{l}(b_{i}+b_{j(i)})=2\sum_{i=1}^{l}b_{i}=0$$

But by Theorem 1.1 of \cite{neumann1978seifert}, $\sum_{i=1}^{l}b_{i}$ is the obstruction term and by assumption is nonzero. Hence there can be no such $f$.

\underline{Case 2:} There are critical fibers.

Let the fillings of $T_{1},\ldots,T_{n}$ be by nontrivially fibered solid tori and the fillings of $T_{n+1},\ldots,T_{l}$ be by trivially fibered solid tori.

Firstly, for $T_{1},\ldots,T_{n}$ we have the diagram:

$$\begin{array}{ccccc}
 &  & (d|_{\partial V_{i}})_{*}\\
 & H_{1}(T_{i}) & \leftarrow & H_{1}(\partial V_{i})\\
(\hat{f}|_{T_{i}})_{*} & \downarrow &  & \downarrow & (d|_{\partial V_{j(i)}}^{-1}\circ\hat{f}|_{T_{i}}\circ d|_{\partial V_{i}})_{*}\\
 & H_{1}(T_{j(i)}) & \leftarrow & H_{1}(\partial V_{j(i)})\\
 &  & (d|_{\partial V_{j(i)}})_{*}
\end{array}$$

Now, in order for $(d|_{\partial V_{j(i)}}^{-1}\circ\hat{f}|_{T_{i}}\circ d|_{\partial V_{i}})_{*}$ to extend into the solid torus, preserve a nontrivial fibration, and be orientation-reversing, according to Corollary 2.2 we must have:

$$(d|_{\partial V_{j(i)}}^{-1}\circ\hat{f}|_{T_{i}}\circ d|_{\partial V_{i}})_{*}=\pm\left[\begin{array}{cc}
1 & 0\\
-1 & -1
\end{array}\right]_{k_{\partial V_{j(i)}}}^{k_{\partial V_{i}}} $$

As the fibration on both $V_{i}$ and $V_{j(i)}$ is a $(-2,1)$ fibration by Proposition 3.1. Hence, we have: 

$$(d|_{\partial V_{i}})_{*}=\left[\begin{array}{cc}
0 & 1\\
1 & 2
\end{array}\right]_{k_{T_{i}}}^{k_{\partial V_{i}}}$$ and $$(d|_{\partial V_{j(i)}})_{*}=\left[\begin{array}{cc}
0 & 1\\
1 & 2
\end{array}\right]_{k_{T_{i}}}^{k_{\partial V_{j(i)}}} $$

So that:

$$(\hat{f}|_{T_{i}})_{*}=\pm\left[\begin{array}{cc}
1 & -1\\
0 & -1
\end{array}\right]_{k_{T_{j(i)}}}^{k_{T_{i}}} $$

That is, for those $V_{i}$ that are nontrivially fibered, $c_{i}=\mp1$. Here again the sign is the same for all.

For $T_{n+1},\ldots,T_{l}$ we proceed as in Case 1, to yield $c_{i}=\mp(b_{i}+b_{j(i)})$ for $i=n+1,\ldots,l$. 

So now,

$$0=\sum_{i=1}^{n+l}c_{i}=\sum_{i=1}^{n}\mp1+\sum_{i=n+1}^{l}\mp(b_{i}+b_{j(i)})=\mp n\mp2b=\pm2e$$

This is twice the Euler class of the bundle which is nonzero. This yields our contradiction

\end{proof}

This proposition establishes the fact that there are no orientation-reversing actions on elliptic manifolds as these have nonzero Euler class.

\section{Manifolds fibering over $S^{2}$}
We apply the results of \cite{peet2018} in the case where the base space of the fibration on the Seifert manifold $M$ has underlying space $S^{2}$. Recall for an action $\varphi:G\rightarrow Diff^{fp}(M)$, there is an induced action $\varphi_{S^{2}}:G_{S^{2}}\rightarrow Diff(S^{2})$. We first consider these possible actions.

\subsection{Finite group actions on $S^{2}$}

By \cite{thurstongeometry}, the possible branching data of a quotient space of $S^{2}$ acted on by a finite group is given by Table 10.1.1. The semidirect product $\circ_{-}$ is defined so that for $H\circ_{-}\mathbb{\mathbb{Z}}_{2}$, the $\mathbb{Z}_{2}$ generator anti-commutes with each element of $H$. Indeed, throughout, this will be the only semidirect product used. If $H$ happens to be abelian, we use $Dih(H)$ instead.

The notation here is such that $rot_{n}^{x}$ is a rotation of order $n$ about the $x$-axis when $S^{2}$ is embedded about the origin in $\mathbb{R}^{3}$, similarly with $rot_{n}^{y},rot_{n}^{z}$. Then $ref^{xy}$ is a reflection in the $x-y$ plane, and again similarly with other reflections. Lastly $rot^{L_{1}},rot^{L_{2}},rot^{L_{3}}$ refer to rotations about lines regarding the rotational symmetry of a tetrahedron, an octahedron, and an icosahedron when inscribed inside $S^{2}$. For more details see \cite{kalliongis2018}. Note that the groups may be given by different names in other sources. For example, $A_{4}\circ_{-}\mathbb{Z}_{2}$ is really $S_{4}$, but we write as a semidirect product for convenience.

These groups form partially ordered sets. We do not expressly show these, but they can be worked out by referring to the generators given.

\begin{rem}
By reference to the generators, it is clear is that any finite group that acts on $S^{2}$ is a subgroup of a finite group that is a semidirect product of a group of orientation-preserving diffeomorphisms and a $\mathbb{Z}_{2}$ generated by an orientation-reversing element. Again, the semidirect product is such that the $\mathbb{Z}_{2}$ generator anti-commutes with all elements of the group of orientation-preserving diffeomorphisms.
\end{rem}

This leads us to consider which of these will satisfy the obstruction condition in Table 1:

\begin{center}
\begin{tabular}{|c|c|c|c|c|}
\hline 
{\footnotesize{}Number} & {\footnotesize{}Underlying Space} & {\footnotesize{}$G$ } & {\footnotesize{}Data} & {\footnotesize{}Generators}\tabularnewline
\hline 
{\footnotesize{}$\begin{array}{c}
1\\
2\\
3\\
4\\
5\\
6\\
7
\end{array}$ } & {\footnotesize{}$S^{2}$ } & {\footnotesize{}$\begin{array}{c}
Trivial\\
\mathbb{Z}_{n}\\
Dih(\mathbb{Z}_{2n})\\
Dih(\mathbb{Z}_{2n+1})\\
A_{4}\\
S_{4}\\
A_{5}
\end{array}$ } & {\footnotesize{}$\begin{array}{c}
()\\
(n,n)\\
(2,2,2n)\\
(2,2,2n+1)\\
(2,3,3)\\
(2,3,4)\\
(2,3,5)
\end{array}$ } & {\footnotesize{}$\begin{array}{c}
id\\
rot_{n}^{z}\\
rot_{2n}^{z},rot_{2}^{y}\\
rot_{2n+1}^{z},rot_{2}^{y}\\
rot_{2}^{z},rot_{3}^{L_{1}}\\
rot_{2}^{z},rot_{3}^{L_{2}}\\
rot_{2}^{z},rot_{3}^{L_{3}}
\end{array}$ }\tabularnewline
\hline 
{\footnotesize{}$\begin{array}{c}
8\\
9\\
10\\
11\\
12\\
13\\
14\\
15\\
16\\
17\\
18\\
19
\end{array}$ } & {\footnotesize{}$D$ } & {\footnotesize{}$\begin{array}{c}
\mathbb{Z}_{2}\\
\mathbb{Z}_{2n}\times\mathbb{Z}_{2}\\
\mathbb{Z}_{4n+2}\\
Dih(\mathbb{Z}_{n})\\
Dih(\mathbb{Z}_{2n})\circ_{-}\mathbb{Z}_{2}\\
Dih(\mathbb{Z}_{2n+1})\circ_{-}\mathbb{Z}_{2}\\
A_{4}\circ_{-}\mathbb{Z}_{2}\\
S_{4}\times\mathbb{Z}_{2}\\
A_{5}\times\mathbb{Z}_{2}\\
Dih(\mathbb{Z}_{2n})\circ_{-}\mathbb{Z}_{2}\\
Dih(\mathbb{Z}_{2n+1})\circ_{-}\mathbb{Z}_{2}\\
A_{4}\times\mathbb{Z}_{2}
\end{array}$ } & {\footnotesize{}$\begin{array}{c}
(;)\\
(2n;)\\
(2n+1;)\\
(;n,n)\\
(;2,2,2n)\\
(;2,2,2n+1)\\
(;2,3,3)\\
(;2,3,4)\\
(;2,3,5)\\
(2;2n)\\
(2;2n+1)\\
(3;2)
\end{array}$ } & {\footnotesize{}$\begin{array}{c}
ref^{xy}\\
rot_{2n}^{z},ref^{xy}\\
rot_{2n+1}^{z}\circ ref^{xy}\\
rot_{n}^{z},ref^{yz}\\
rot_{2n}^{z},rot_{2}^{y},ref^{yz}\\
rot_{2n+1}^{z},rot_{2}^{y},ref^{yz}\\
rot_{2}^{z},rot_{3}^{L_{1}},ref^{yz}\\
rot_{2}^{z},rot_{3}^{L_{2}},ref^{xy}\\
rot_{2}^{z},rot_{3}^{L_{3}},ref^{xy}\\
rot_{2n}^{z},rot_{2}^{y},ref^{xz}\\
rot_{2n+1}^{z},rot_{2}^{y},ref^{xz}\\
rot_{2}^{z},rot_{3}^{L_{1}},ref^{xy}
\end{array}$ }\tabularnewline
\hline 
{\footnotesize{}$\begin{array}{c}
20\\
21
\end{array}$ } & {\footnotesize{}$\mathbb{P}^{2}$ } & {\footnotesize{}$\begin{array}{c}
\mathbb{Z}_{2}\\
\mathbb{Z}_{2n}
\end{array}$ } & {\footnotesize{}$\begin{array}{c}
()\\
(n)
\end{array}$ } & {\footnotesize{}$\begin{array}{c}
rot_{2}^{z}\circ ref^{xy}\\
rot_{2n}^{z}\circ ref_{xy}
\end{array}$ }\tabularnewline
\hline 
\end{tabular}
\end{center}

\begin{center}
{Table 1: Orbit numbers of finite group actions
on $S^{2}$.}
\par\end{center}

\begin{center}
\begin{tabular}{|c|c|c|c|c|c|}
\hline 
{\footnotesize{}Number} & {\footnotesize{}$G$ } & {\footnotesize{}Orbit Numbers of non-regular points} & {\footnotesize{}$LCM$ } & {\footnotesize{}$|G|/LCM$ } & {\footnotesize{}OC Satisfied?}\tabularnewline
\hline 
{\footnotesize{}$\begin{array}{c}
1\\
2\\
3\\
4\\
5\\
6\\
7
\end{array}$ } & {\footnotesize{}$\begin{array}{c}
Trivial\\
\mathbb{Z}_{n}\\
Dih(\mathbb{Z}_{2n})\\
Dih(\mathbb{Z}_{2n+1})\\
A_{4}\\
S_{4}\\
A_{5}
\end{array}$ } & {\footnotesize{}$\begin{array}{c}
1\\
1,1\\
2,2n,2n\\
2,2n+1,2n+1\\
4,4,6\\
6,8,12\\
6,10,15
\end{array}$ } & {\footnotesize{}$\begin{array}{c}
1\\
n\\
2n\\
4n+2\\
6\\
12\\
30
\end{array}$ } & {\footnotesize{}$\begin{array}{c}
1\\
1\\
2\\
1\\
2\\
2\\
1
\end{array}$} & {\footnotesize{}$\begin{array}{c}
\textrm{all }b\\
\textrm{all }b\\
b\textrm{ even}\\
\textrm{all }b\\
b\textrm{ even}\\
b\textrm{ even}\\
\textrm{all }b
\end{array}$}\tabularnewline
\hline 
{\footnotesize{}$\begin{array}{c}
8\\
9\\
10\\
11\\
12\\
13\\
14\\
15\\
16\\
17\\
18\\
19
\end{array}$ } & {\footnotesize{}$\begin{array}{c}
\mathbb{Z}_{2}\\
\mathbb{Z}_{2n}\times\mathbb{Z}_{2}\\
\mathbb{Z}_{4n+2}\\
Dih(\mathbb{Z}_{n})\\
Dih(\mathbb{Z}_{2n})\circ_{-}\mathbb{Z}_{2}\\
Dih(\mathbb{Z}_{2n+1})\circ_{-}\mathbb{Z}_{2}\\
A_{4}\circ\mathbb{Z}_{2}\\
S_{4}\times\mathbb{Z}_{2}\\
A_{5}\times\mathbb{Z}_{2}\\
Dih(\mathbb{Z}_{2n})\circ_{-}\mathbb{Z}_{2}\\
Dih(\mathbb{Z}_{2n+1})\circ_{-}\mathbb{Z}_{2}\\
A_{4}\times\mathbb{Z}_{2}
\end{array}$ } & {\footnotesize{}$\begin{array}{c}
1\\
2,n\\
2,2n+1\\
1,1\\
2,2n,2n\\
2,2n+1,2n+1\\
4,4,6\\
6,8,12\\
6,10,15\\
2,4n\\
2,4n+2\\
6,8
\end{array}$ } & {\footnotesize{}$\begin{array}{c}
2\\
2n\\
4n+2\\
2n\\
4n\\
8n+4\\
12\\
24\\
60\\
4n\\
4n+2\\
12
\end{array}$ } & {\footnotesize{}$\begin{array}{c}
1\\
2\\
1\\
1\\
2\\
1\\
2\\
2\\
2\\
2\\
2\\
2
\end{array}$ } & {\footnotesize{}$\begin{array}{c}
\textrm{all }b\\
b\textrm{ even}\\
\textrm{all }b\\
\textrm{all }b\\
b\textrm{ even}\\
\textrm{all }b\\
b\textrm{ even}\\
b\textrm{ even}\\
b\textrm{ even}\\
b\textrm{ even}\\
b\textrm{ even}\\
b\textrm{ even}
\end{array}$ }\tabularnewline
\hline 
{\footnotesize{}$\begin{array}{c}
20\\
21
\end{array}$ } & {\footnotesize{}$\begin{array}{c}
\mathbb{Z}_{2}\\
\mathbb{Z}_{2n}
\end{array}$ } & {\footnotesize{}$\begin{array}{c}
-\\
2
\end{array}$ } & {\footnotesize{}$\begin{array}{c}
1\\
n
\end{array}$ } & {\footnotesize{}$\begin{array}{c}
2\\
2
\end{array}$ } & {\footnotesize{}$\begin{array}{c}
b\textrm{ even}\\
b\textrm{ even}
\end{array}$ }\tabularnewline
\hline 
\end{tabular}
{\footnotesize\par}
\par\end{center}

\begin{center}
{Table 2: Orbit numbers of finite group actions
on $S^{2}$.}
\par\end{center}

\begin{rem}

Note that for all actions with induced actions as above, the obstruction condition will be satisfied if the obstruction term is even, but there could actions that will not satisfy the obstruction condition if the obstruction term is odd. One such action is exhibited in \cite{peet2018}.
\end{rem}

\subsection{Manifolds fibering over $S^{2}$}
We now prove a general result that will set up the group structure for the groups acting on manifolds fibering over an orbifold with underlying space $S^{2}$. Throughout this section we assume normalized form for Seifert invariants.

\begin{prop} Let $M=(0,o_{1}|(q_{1},p_{1}),\ldots,(q_{n},p_{n}),(1,b))$ and $\varphi:G\rightarrow Diff_{+}^{fp}(M)$ be a finite action that satisfies the obstruction condition. Then $G$ is isomorphic to a subgroup of $(\mathbb{Z}_{m}\times G_{S^{2}+})\circ_{-}\mathbb{Z}_{2}$ for some $m\in\mathbb{N}$ and $G_{S^{2}+}$ is the orientation-preserving subgroup of the induced action $\varphi_{S^{2}}:G_{S^{2}}\rightarrow Diff(S^{2})$.
\end{prop}

\begin{proof}
As $\varphi:G\rightarrow Diff_{+}^{fp}(M)$ satisfies the obstruction condition, we can restrict to the action $\hat{\varphi}:G\rightarrow Diff(S^{1})\times Diff(F)$.

Now, consider $\hat{\varphi}_{S^{1}}(G)$, the projection onto $Diff(S^{1})$. So then $\hat{\varphi}_{S^{1}}(G)$ is a subgroup of $Dih(\mathbb{Z}_{m})\cong\mathbb{Z}_{m}\circ_{-}\mathbb{Z}_{2}$ for some $m$. 

Also, $\hat{\varphi}_{F}(G)$ the projection onto $Diff(F)$ will be a subgroup of $\hat{\varphi}_{F}(G)_{+}\circ_{-}\mathbb{Z}_{2}$ by the remark above, where $\hat{\varphi}_{F}^{+}(G)_{+}$ is the orientation-preserving subgroup. 

So now, $\hat{\varphi}(G)\subset\hat{\varphi}_{S^{1}}(G)\times\hat{\varphi}_{F}(G)\subset(\mathbb{Z}_{m}\circ_{-}\mathbb{Z}_{2})\times(\hat{\varphi}_{F}(G)_{+}\circ_{-}\mathbb{Z}_{2})$

But, $\hat{\varphi}(G)$ is orientation-preserving. Hence, we consider the orientation-preserving subgroup of $(\mathbb{Z}_{m}\circ_{-}\mathbb{Z}_{2})\times(\hat{\varphi}_{F}(G)_{+}\circ_{-}\mathbb{Z}_{2})$.

Note that $g\in((\mathbb{Z}_{m}\circ_{-}\mathbb{Z}_{2})\times(\hat{\varphi}_{F}(G)\circ_{-}\mathbb{Z}_{2}))_{+}$ if and only if $g=(g_{1},g_{2})$ or $g=(g_{1}z_{1},g_{2}z_{2})$ for $(g_{1},g_{2})\in\mathbb{Z}_{m}\times\hat{\varphi}_{F}(G)_{+}$ and $z_{1},z_{2}$ are respective generators of the two $\mathbb{Z}_{2}$ components. It therefore follows that $((\mathbb{Z}_{m}\circ_{-}\mathbb{Z}_{2})\times(\hat{\varphi}_{F}(G)_{+}\circ_{-}\mathbb{Z}_{2}))_{+}=(\mathbb{Z}_{m}\times\hat{\varphi}_{F}(G)_{+})\circ_{-}\mathbb{Z}_{2}$ where the $\mathbb{Z}_{2}$ is generated by $z=(z_{1},z_{2})$, and the semidirect product is defined by $z(g_{1},g_{2})z^{-1}=(g_{1}^{-1},g_{2}^{-1})$.

Now, $\hat{\varphi}_{F}(G)_{+}\cong G_{S^{2}+}$ and $\hat{\varphi}(G)\cong G$ so that $G\subset(\mathbb{Z}_{m}\times G_{S^{2}+})\circ_{-}\mathbb{Z}_{2}$.

\end{proof}

This result essentially states that we need only check that the obstruction condition is satisfied and calculate the possible orientation-preserving subgroup of the induced action $\varphi_{S^{2}}:G_{S^{2}}\rightarrow Diff(S^{2})$. This we can do by reference to the Tables 1 and 2.

We now proceed to consider the individual cases for the number of critical fibers. For each proof the construction set out in \cite{peet2018} provides the converse.

\subsection{One critical fiber}

We now consider the case where there is only one critical fiber.

\begin{cor} Let $M=(0,o_{1}|(q,p),(1,b))$. There exists a finite action $\varphi:G\rightarrow Diff_{+}^{fp}(M)$ if and only if $G$ is isomorphic to a subgroup of $Dih(\mathbb{Z}_{m}\times\mathbb{Z}_{n})$ for some $m,n\in\mathbb{N}$. 
\end{cor}

\begin{proof}
Note that the induced action $\varphi_{S^{2}}:G_{S^{2}}\rightarrow Diff(S^{2})$ must fix a point. By Tables 1 and 2, we can assume that this is of the form of action 11. This action satisfies the obstruction condition for any $b$. Hence, $G_{S^{2}+}\cong\mathbb{Z}_{n}$ for some $n\in\mathbb{N}$. Then by Proposition 4.1, $G$ is isomorphic to a subgroup of $(\mathbb{Z}_{m}\times\mathbb{Z}_{n})\circ_{-}\mathbb{Z}_{2}=Dih(\mathbb{Z}_{m}\times\mathbb{Z}_{n})$.  

\end{proof}

\subsection{Two critical fibers}

Now consider two critical fibers. Firstly, when the respective normalized fillings are not equal.

\begin{cor} Let $M=(0,o_{1}|(q_{1},p_{1}),(q_{2},p_{2}),(1,b))$ with $(q_{1},p_{1})\neq(q_{2},p_{2})$. There exists a finite action $\varphi:G\rightarrow Diff_{+}^{fp}(M)$ if and only if $G$ is isomorphic to a subgroup of $Dih(\mathbb{Z}_{m}\times\mathbb{Z}_{n})$ for some $m,n\in\mathbb{N}$. 
\end{cor}

\begin{proof}
Note that the induced action $\varphi_{S^{2}}:G_{S^{2}}\rightarrow Diff(S^{2})$ must fix two points. By Tables 1 and 2, we again assume the form of action 11. This action satisfies the obstruction condition for any $b$. Hence, $G_{S^{2}+}\cong\mathbb{Z}_{n}$ for some $n\in\mathbb{N}$. Then by Proposition 4.1, $G$ is isomorphic to a subgroup of $(\mathbb{Z}_{m}\times\mathbb{Z}_{n})\circ_{-}\mathbb{Z}_{2}=Dih(\mathbb{Z}_{m}\times\mathbb{Z}_{n})$.  
\end{proof}

Now consider when the fillings of the two critical fibers are equal.

\begin{cor} Let $M=(0,o_{1}|(q,p),(q,p),(1,b))$ with $b$ even. There exists a finite action $\varphi:G\rightarrow Diff_{+}^{fp}(M)$ if and only if $G$ is isomorphic to a subgroup of $(\mathbb{Z}_{m}\times Dih(\mathbb{Z}_{n}))\circ_{-}\mathbb{Z}_{2}$  for some $m,n\in\mathbb{N}$. 
\end{cor}

\begin{proof}
 We assume that the induced action $\varphi_{S^{2}}:G_{S^{2}}\rightarrow Diff(S^{2})$ exchanges two points referring to the critical fibers. Otherwise, we apply Corollary 4.3. Given that two points are exchanged, we consult the Tables to note that we can assume that $\varphi_{S^{2}}$ is in the form of actions 12/13 or 17/18. The obstruction condition will be satisfied for each of these as we assume that $b$ is even. Then in either case, $G_{S^{2}+}\cong Dih(\mathbb{Z}_{n})$ and by Proposition 4.1, $G$ is isomorphic to a subgroup of $(\mathbb{Z}_{m}\times G_{S^{2}+})\circ_{-}\mathbb{Z}_{2}\cong(\mathbb{Z}_{m}\times Dih(\mathbb{Z}_{n}))\circ_{-}\mathbb{Z}_{2}$.  
\end{proof}

\begin{rem}
Note that $M=(0,o_{1}|(q,p),(q,p),(1,b))$ with $b$ even is simply $S^{2}\times S^{1}$.  $M=(0,o_{1}|(q,p),(q,p),(1,b))$ with $b$ odd is a Lens space and so as an exception to our results, we refer the reader to \cite{kalliongis2002geometric} for a classification of finite actions on these manifolds.
\end{rem}

\subsection{Three critical fibers}

We now move on to having three critical fibers and break into the three possible scenarios: that they all have different fillings; that two have the same fillings; and that they all have the same filling.

\begin{cor} Let $M=(0,o_{1}|(q_{1},p_{1}),(q_{2},p_{2}),(q_{3},p_{3}),(1,b))$ with $(q_{i},p_{i})$ all different. There exists a finite action $\varphi:G\rightarrow Diff_{+}^{fp}(M)$ if and only if $G$ is isomorphic to a subgroup of $Dih(\mathbb{Z}_{m})$ for some $m\in\mathbb{N}$.
\end{cor}

\begin{proof}
Note that the induced action $\varphi_{S^{2}}:G_{S^{2}}\rightarrow Diff(S^{2})$ must fix three points. By Tables 1 and 2, the only such induced action is the trivial action 1, that is $G_{S^{2}}$ is the trivial group. This action trivially satisfies the obstruction condition for any $b$. 

Hence, by Proposition 4.1, $G$ is a subgroup of $(\mathbb{Z}_{m}\times G_{S^{2}+})\circ_{-}\mathbb{Z}_{2}\cong\mathbb{Z}_{m}\circ_{-}\mathbb{Z}_{2}=Dih(\mathbb{Z}_{m})$. 
  
\end{proof}

\begin{cor} Let $M=(0,o_{1}|(q_{1},p_{1}),(q,p),(q,p),(1,b))$ with $(q_{1},p_{1})\neq(q,p)$. There exists a finite action $\varphi:G\rightarrow Diff_{+}^{fp}(M)$ if and only if $G$ is isomorphic to a subgroup of $Dih(\mathbb{Z}_{m}\times\mathbb{Z}_{2})$ for some $m\in\mathbb{N}$.
\end{cor}

\begin{proof}
Note that the induced action $\varphi_{S^{2}}:G_{S^{2}}\rightarrow Diff(S^{2})$ must fix a point and at most exchange two others. By Tables 1 and 2, the only such action is of the form of action 11 with $n=2$. This action satisfies the obstruction condition for any $b$. So $G_{S^{2}+}\cong\mathbb{Z}_{2}$. Hence by Proposition 4.1, $G$ is a subgroup of $(\mathbb{Z}_{m}\times\mathbb{Z}_{2})\circ_{-}\mathbb{Z}_{2}=Dih(\mathbb{Z}_{m}\times\mathbb{Z}_{2})$. 
\end{proof}

\begin{cor} Let $M=(0,o_{1}|(q,p),(q,p),(q,p),(1,b))$. There exists a finite action $\varphi:G\rightarrow Diff_{+}^{fp}(M)$ if and only if $G$ is isomorphic to a subgroup of $(\mathbb{Z}_{m}\times Dih(\mathbb{Z}_{3}))\circ_{-}\mathbb{Z}_{2}$.
\end{cor}

\begin{proof}
We assume that $\varphi_{S^{2}}:G_{S^{2}}\rightarrow Diff(S^{2})$ exchanges three points, else apply Corollary 4.5. or Corollary 4.6. So now by Tables 1 and 2 we can assume that $\varphi_{S^{2}}:G_{S^{2}}\rightarrow Diff(S^{2})$ is of the form of action 13 with $n=1$. This action satisfies the obstruction condition for any $b$ and $G_{S^{2}+}\cong Dih(\mathbb{Z}_{3})$. Hence, by Proposition 4.1, $G$ is isomorphic to a subgroup of $(\mathbb{Z}_{m}\times Dih(\mathbb{Z}_{3}))\circ_{-}\mathbb{Z}_{2}$.  
\end{proof}

\subsection{No critical fibers}

In the case where there are no critical fibers, we note that there are no restrictions on $\varphi_{S^{2}}:G_{S^{2}}\rightarrow Diff(S^{2})$. Hence we cannot guarantee that the obstruction condition will be satisfied unless $b$ is even. In such a case the group will be a subgroup of a group of the form $(\mathbb{Z}_{m}\times H)\circ_{-}\mathbb{Z}_{2}$ where $H$ is a group from the list of groups that act orientation-preservingly on $S^{2}$. Note, however that once again, these manifolds are Lens spaces of the form $L(b,1)$ and we again refer the reader to \cite{kalliongis2002geometric}.

\subsection{Manifolds fibering over $\mathbb{P}^{2}$}

We here apply the results of \cite{peet2018finite} to yield the following result: 

\begin{cor} Let $M=(1,n_{2}|(q,p),(1,b))$. There exists a finite action $\varphi:G\rightarrow Diff_{+}^{fp}(M)$ if and only if $G$ is isomorphic to a subgroup of $\mathbb{Z}_{2}\times Dih(\mathbb{Z}_{n})$ for some $n\in\mathbb{N}$.
\end{cor}

\begin{proof}
Let $\tilde{M}=(0,o_{1}|(q,p),(q,p),(1,2b))$ be the orientable base space double cover of $M$. According to \cite{peet2018finite}, we consider a corresponding finite action $\tilde{\varphi}:G\rightarrow Diff_{+}^{fop}(\tilde{M})$ that commutes with the covering translation $\tau:\tilde{M}\rightarrow\tilde{M}$. 

Now note that the induced action $\tilde{\varphi}_{S^{2}}:G_{S^{2}}\rightarrow Diff(S^{2})$ can exchange two points but must be orientation-preserving as $\tilde{\varphi}:G\rightarrow Diff_{+}^{fop}(\tilde{M})$ is fiber-orientation-preserving. We can then assume that $\varphi_{S^{2}}$ is in the form of actions 3/4. Then, $G_{S^{2}}\cong Dih(\mathbb{Z}_{n})$ for some $n\in\mathbb{N}$. Again by Table 2, it will satisfy the obstruction condition as $2b$ is even. 

So now apply the results of \cite{peet2018finite} to note that there is a restricted action $\widehat{\tilde{\varphi}}:G\rightarrow Diff(\hat{\tilde{M}})$ and product structure $k:S^{1}\times F\rightarrow\widehat{\tilde{M}}$ ($F$ is in fact an annulus) so such that $(k^{-1}\circ\widehat{\tilde{\varphi}}(g)\circ k)(u,x)=(\epsilon(g)u,\widehat{\tilde{\varphi}}_{2}(g)(x))$ for $\epsilon(g)=\pm1$. We then note that $\widehat{\tilde{\varphi}}_{2}(G)\cong G_{S^{2}}\cong Dih(\mathbb{Z}_{n})$. Hence, $\widehat{\tilde{\varphi}}(G)\cong G$ is isomorphic to a subgroup of $\widehat{\tilde{\varphi}}_{1}(G)\times\widehat{\tilde{\varphi}}_{2}(G)\cong\mathbb{Z}_{2}\times Dih(\mathbb{Z}_{n})$. Again, our construction set out in \cite{peet2018} provides the converse. 
\end{proof}

\section{Elliptic 3-manifolds}
Recall that elliptic 3-manifolds are Seifert manifolds where $\chi_{orb}(B)>0$ and the Euler class of the Seifert bundle is nonzero. \cite{scott1983geometries} By \cite{thurstongeometry}, the orbifolds without boundary that have positive orbifold Euler characteristic are:

$$S^{2},S^{2}(q_{1}),S^{2}(q_{1},q_{2}),S^{2}(2,2,q),\mathbb{P}^{2}(q),S^{2}(2,3,3),S^{2}(2,3,4),S^{2}(2,3,5)$$

We note that by Proposition 3.2, all fiber-preserving actions on elliptic manifolds are orientation-preserving as the Euler class must be nonzero. Hence we can break down the possible base spaces and apply the results of the previous sections. In each subsection, suppose that we have a finite action $\varphi:G\rightarrow Diff^{fp}(M)$.

\subsection{Base space $S^{2}$}

These manifolds are lens spaces fibered without critical fibers. By \cite{kalliongis2002geometric}, these are of the form $L(p,q)$ where $q=\pm1\textrm{(mod }p)$. 

By Remark 2, we can only certainly work with even obstruction condition and in which case the lens space is constructed by two $(b,1)$ fillings of $S^{1}\times A$. We then calculate:

$$\left[\begin{array}{cc}
-1 & b\\
0 & 1
\end{array}\right]\left[\begin{array}{cc}
1 & 0\\
0 & -1
\end{array}\right]\left[\begin{array}{cc}
-1 & b\\
0 & 1
\end{array}\right]=\left[\begin{array}{cc}
1 & -2b\\
0 & -1
\end{array}\right]$$

Thus we have the lens spaces $L(2b,1)$ for nonzero $b\in\mathbb{Z}$. 

So now we apply Section 4 to state that the group $G$ will be a subgroup of a group of the form $(\mathbb{Z}_{m}\times H)\circ_{-}\mathbb{Z}_{2}$ where $H$ is a group from the list of groups that act orientation-preservingly on $S^{2}$ and $m\in\mathbb{N}$.

\subsection{Base space $S^{2}(q)$}

These manifolds are again lens spaces, but fibered with one critical fiber. All lens spaces can be given such a fibration except those of the form $L(p,q)$ where $q=\pm1\textrm{(mod }p)$ mentioned above. This follows from fibering one solid torus side of the Heegaard torus trivially and inducing a fibration on the other side. 

We can now apply Corollary 4.2. to find that the group $G$ is a subgroup of $Dih(\mathbb{Z}_{m}\times\mathbb{Z}_{n})$ for $m,n\in\mathbb{N}$.

\subsection{Base space $S^{2}(q_{1},q_{2})$}

So $M=(0,o_{1}|(p_{1},q_{1}),(p_{2},q_{2}),(1,b))$. Once again, these manifolds are lens spaces, but now fibered with two critical fibers. All lens spaces can be fibered in this way.

We first consider $(p_{1},q_{1})\neq(p_{2},q_{2})$. Then $G$ is a subgroup of $Dih(\mathbb{Z}_{m}\times\mathbb{Z}_{n})$ for $m,n\in\mathbb{N}$ by Corollary 4.3.

If $(p_{1},q_{1})=(p_{2},q_{2})$ then our results only apply in the case where $b$ is even, in which case the manifold is not elliptic by Remark 3.

\subsection{Base space $S^{2}(2,2,q)$}

So $M=(0,o_{1}|(q,p),(2,1),(2,1),(1,b))$. These manifolds are now prism manifolds fibered longitudinally. We split into the two cases:

\underline{Case 1:} $(q,p)=(2,1)$

In this case we apply Corollary 4.7 to yield that $G$ is a subgroup of $(\mathbb{Z}_{m}\times Dih(\mathbb{Z}_{3}))\circ_{-}\mathbb{Z}_{2}$ for $m\in\mathbb{N}$.

\underline{Case 2:} $(q,p)\neq(2,1)$

In this case we instead apply Corollary 4.6 to yield that $G$ is a subgroup of $Dih(\mathbb{Z}_{m}\times\mathbb{Z}_{2})$ for $m\in\mathbb{N}$.

\subsection{Base space $\mathbb{P}^{2}(p)$}

These manifolds are again prism manifolds but fibered meridianally.

We apply Corollary 4.8 to yield that the group $G$ is a subgroup of $\mathbb{Z}_{2}\times Dih(\mathbb{Z}_{n})$ for some $n\in\mathbb{N}$.

\subsection{Base space $S^{2}(2,3,3)$}

In this case, $M=(0,o_{1}|(2,1),(3,p_{1}),(3,p_{2}),(1,b))$ for $p_{1}=1,2$ and $p_{2}=1,2$. We hence break into the two possible cases:

\underline{Case 1:} $p_{1}=p_{2}$

In this case we apply Corollary 4.6 to yield that $G$ is a subgroup of $(\mathbb{Z}_{m}\times Dih(\mathbb{Z}_{3}))\circ_{-}\mathbb{Z}_{2}$ for $m\in\mathbb{N}$.

\underline{Case 2:} $p_{1}\neq p_{2}$

In this case we instead apply Corollary 4.5 to yield that $G$ is a subgroup of $Dih(\mathbb{Z}_{m})$ for some $m\in\mathbb{N}$. 

\subsection{Base spaces $S^{2}(2,3,4)$ and $S^{2}(2,3,5)$}

In both of these cases $M=(0,o_{1}|(2,1),(3,p_{1}),(q_{2},p_{2}),(1,b))$ for ${q_{2}\neq2,3}$. Hence we apply apply Corollary 4.7 to yield that $G$ is a subgroup of $Dih(\mathbb{Z}_{m})$ for some $m\in\mathbb{N}$. 

\section{Quotient spaces}

We now consider the quotient spaces under these constructed actions.

\subsection{General outline of construction}

We first note that an orientation and fiber-preserving action on a fibered torus will have quotient type either a torus or a $S^{2}(2,2,2,2)$. This follows from \cite{scott1983geometries} and the fact that $S^{2}(2,3,6)$, $S^{2}(3,3,3)$, and $S^{2}(2,4,4)$ cannot be Seifert fibered. We then consider the quotient of $\hat{M}\cong S^{1}\times F$ under a product action and the stabilizers of the boundary tori. Here, $F$ will be a disc with holes. There will be then glued in either a solid torus with exceptional core or a Conway ball.

The main part is to establish what form the quotients of $\hat{M}$ and each $V_{i}$ will be, and then how the gluing maps look under the projection.

Formally, for a representation $M=(0,o_{2}|(q_{1},p_{1}),\ldots,(q_{n},p_{n}))$, we take an action $\varphi:G\rightarrow Diff_{+}^{fp}(M)$ that restricts to an action $\hat{\varphi}:G\rightarrow Diff_{+}^{fp}(\hat{M})$ which leaves some fibering product structure $k:S^{1}\times F\rightarrow\hat{M}$ invariant and which extends over the fillings of a collection of fibered solid tori $X=V_{1}\cup\ldots\cup V_{n}$. We denote $\varphi_{X}:G\rightarrow Diff_{+}^{fp}(X)$ to be the restricted action on $X$.

We then let $\hat{p}:\hat{M}\rightarrow\hat{M}/\hat{\varphi}$ and $p_{X}:X\rightarrow X/\varphi_{X}$ be the quotient maps.

We then have the diagram:

$$\begin{array}{ccccc}
 &  & d|_{\partial V_{i}}\\
 & T_{i} & \leftarrow & \partial V_{i}\\
\hat{p}|_{T_{i}} & \downarrow &  & \downarrow & (p_{X})|_{\partial V_{i}}\\
 & T_{i}/Stab_{\hat{\varphi}}(T_{i}) & \leftarrow & \partial V_{i}/Stab_{\varphi_{X}}(\partial V_{i})\\
 &  & d'|_{\partial V_{i}/Stab_{\varphi_{X}}(\partial V_{i})}
\end{array}$$

We hence need to find the following:

• $\hat{M}/\hat{\varphi} $

• $V_{i}/Stab_{\varphi_{X}}(V_{i})$

• $d'|_{\partial V_{i}/Stab_{\varphi_{X}}(\partial V_{i})}$

\subsection{$\hat{M}/\hat{\varphi}$}

We first consider actions constructed via the method of \cite{peet2018} that are fiber-orientation-preserving. For this section we consider $F$ in the more general setting as any orientable surface with boundary.

\begin{lem} Let $\hat{\varphi}_{S^{1}\times F}:G\rightarrow Diff_{+}(S^{1})\times Diff_{+}(F)$ be a finite group action such that no element leaves an isolated fiber invariant. Then $(S^{1}\times F)/\hat{\varphi}_{S^{1}\times F}$ is a trivially fibered Seifert 3-manifold with fibering product structure $S^{1}\times(F/(\hat{\varphi}_{S^{1}\times F})_{F})$.
\end{lem}

\begin{proof}
It is clear that $(S^{1}\times F)/\hat{\varphi}_{S^{1}\times F}$ is a trivially fibered Seifert 3-manifold. It remains to show that it has the fibering product structure $S^{1}\times(F/(\hat{\varphi}_{S^{1}\times F})_{F})$. We examine the diagram:

$$\begin{array}{cccc}
 &  & proj_{F}\\
 & S^{1}\times F & \rightarrow & F\\
p_{\varphi_{S^{1}\times F}} & \downarrow\\
 & (S^{1}\times F)/\hat{\varphi}_{S^{1}\times F}=S^{1}\times F' & \rightarrow & F'\\
 &  & proj_{F'}
\end{array}$$

Now, $p_{\varphi_{S^{1}\times F}}:S^{1}\times F\rightarrow S^{1}\times F'$ can be chosen so that $p_{\varphi_{S^{1}\times F}}(u,x)=(p_{1}(u,x),p_{2}(x))$. 

So now, we have $proj_{F'}(p(u,x))=p_{2}(x)$ and $p_{2}$ is the covering map for $(\hat{\varphi}_{S^{1}\times F})_{F}$. Hence $F'=F/(\hat{\varphi}_{S^{1}\times F})_{F}$.

\end{proof}

\begin{rem}

As $\hat{\varphi}_{S^{1}\times F}:G\rightarrow Diff_{+}(S^{1})\times Diff_{+}(F)$ will be extended over some fillings, if it does happen to leave some isolated fibers invariant, then we can simply drill out these fibers and restrict to $\hat{\varphi}'_{S^{1}\times F}:G\rightarrow Diff_{+}(S^{1})\times Diff_{+}(F')$ and then consider the resultant torus boundaries to be filled according to a $(1,0)$ filling. Therefore, for our purposes, we can without loss of generality assume that our action $\hat{\varphi}_{S^{1}\times F}:G\rightarrow Diff_{+}(S^{1})\times Diff_{+}(F)$ does not leave any isolated fibers invariant and so the previous lemma holds.

\end{rem}

We now allow the fibers to be reversed.

\begin{lem} Let $\hat{\varphi}_{S^{1}\times F}:G\rightarrow Diff(S^{1})\times Diff(F)$ be a finite, orientation-preserving group action so that $(\hat{\varphi}_{S^{1}\times F})_{+}:G_{+}\rightarrow Diff_{+}(S^{1})\times Diff_{+}(F)$ is such that no element leaves an isolated fiber invariant. Then any element that reverses the orientation on both components will induce some product involution $f=(f_{1},f_{2})$ of $S^{1}\times(F/((\hat{\varphi}_{S^{1}\times F})_{+}{}_{F})$ that also reverses the orientation on both components. Then $(S^{1}\times F)/\hat{\varphi}_{S^{1}\times F}$ is found by taking $I\times(F/(\hat{\varphi}_{S^{1}\times F})_{+}{}_{F})$ and identifying $(i,x)$ with $(i,f_{2}(x))$ for $i=0,1$ and leaving exceptional sets of order 2 as properly embedded arcs or circles according to the fixed point set of $f_{2}$.
\end{lem}

\begin{proof}
If $g_{-}$ is an element of $G$, so that $\hat{\varphi}_{S^{1}\times F}(g_{-})$ reverses the orientation on both components, then we have some $f:S^{1}\times(F/((\hat{\varphi}_{S^{1}\times F})_{+}{}_{F})\rightarrow S^{1}\times(F/((\hat{\varphi}_{S^{1}\times F})_{+}{}_{F})$ so that $p_{(\hat{\varphi}_{S^{1}\times F})_{+}}\circ\varphi(g_{-})=f\circ p_{(\hat{\varphi}_{S^{1}\times F})_{+}}$. 

To see that $f$ is an involution requires only the observation that $G_{+}$ is an index two subgroup of $G$.

To see that it is a product, we note that if it does not preserve the product structure $S^{1}\times(F/((\hat{\varphi}_{S^{1}\times F})_{+}{}_{F})$, then $g_{-}$ cannot preserve the product structure $S^{1}\times F$. Hence $f$ is a product reversing the orientation on both components. 

The result therefore follows.

\end{proof}

\begin{cor} Let $F$ be a genus $0$ surface with boundary. Let$\hat{\varphi}_{S^{1}\times F}:G\rightarrow Diff(S^{1})\times Diff(F)$ be a finite, orientation-preserving group action so that $(\hat{\varphi}_{S^{1}\times F})_{+}:G_{+}\rightarrow Diff_{+}(S^{1})\times Diff_{+}(F)$ is such that no element leaves an isolated fiber invariant. Then any element that reverses the orientation on both components will induce some product involution $f=(f_{1},f_{2})$ of $S^{1}\times(F/((\hat{\varphi}_{S^{1}\times F})_{+}{}_{F})$ that also reverses the orientation on both components. Then $(S^{1}\times F)/\hat{\varphi}_{S^{1}\times F}$ is a ball  $B$ less a disjoint collection of balls and solid tori in the interior of $B$ with exceptional sets of order 2 as properly embedded arcs or circles according to the fixed point set of $f_{2}$.
\end{cor}

\begin{proof}
From Lemma 6.2, we note that $S^{1}\times(F/((\hat{\varphi}_{S^{1}\times F})_{+}{}_{F})$ is simply another genus $0$ surface with boundary cross $I$. It then follows that the boundary identification will fold $S^{1}\times(F/((\hat{\varphi}_{S^{1}\times F})_{+}{}_{F})$ up to a ball with removed interior balls and solid tori and exceptional sets of order 2 as properly embedded arcs or circles according to the fixed point set of $f_{2}$. 
\end{proof}

\begin{exmp}
Consider $F \times S^{1}$ where $F$ is a disc with three discs removed. Then take a $Dih(\mathbb{Z}_2)$-action on $F \times S^{1}$ generated by $g_{1}$: an order $2$ rotation on $F$ fixing two of the boundary components and exchanging the other two with no rotation in the $S^{1}$ component, and $g_{2}$: the antipodal map on $F$ and a reflection on $S^{1}$.
Then $(F \times S^{1})/\langle g_{1} \rangle$ is an annulus cross $s^{1}$. $g_{2}$ induces an involution on this space consisting of the antipodal map on the annulus and a reflection in the $S^{1}$ component. This quotients to a ball with no interior balls removed and no exceptional set.
\end{exmp}

\begin{exmp}
Now consider again $F \times S^{1}$ where $F$ is a disc with three discs removed. This time take a $Dih(\mathbb{Z}_2)$-action on $F \times S^{1}$ generated by $g_{1}$: an order $2$ rotation in the $S^{1}$ component and the identity on $F$ and $g_{2}$: a reflection on $F$ that leaves two boundary components invariant, exchanging the other two, and a reflection in the $S^{1}$ component.
Here $(F \times S^{1})/\langle g_{1} \rangle$ is homeomorphic to $F \times S^{1}$ and $g_{2}$ induces the same map on the quotient space. This then quotients to the following space:

\begin{figure}[ht]
\centering
\includegraphics[height=5cm]{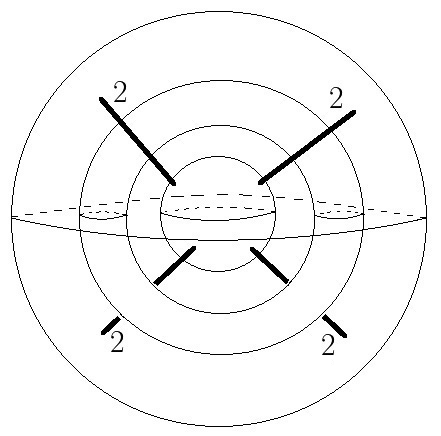}
\caption{Quotient under the action of Example 6.2 }
\end{figure}

\end{exmp}

\subsection{$V_{i}/Stab_{\varphi_{X}}(V_{i})$}

We begin by assuming that the action preserves the orientation of the fibers and note that the filling is of a fibered solid torus where the critical fiber is also an exceptional set. By \cite{kalliongis1991symmetries} the action of the stabilizer on $V_{i}$ will be a $\mathbb{Z}_{m}\times\mathbb{Z}_{l}$-action where $m$ divides $l$ with generators $\varphi(g_{1})(u,v)=(e^{2a\pi i}u,e^{2b\pi i}v)$ and $\varphi(g_{2})(u,v)=(e^{2c\pi i}u,e^{2d\pi i}v)$ where $a=\frac{a_{1}}{a_{2}},b=\frac{b_{1}}{b_{2}},c=\frac{c_{1}}{c_{2}},d=\frac{d_{1}}{d_{2}}$ are rational numbers. The quotient will then be a solid torus with an exceptional core of order $k$, where $m|k|l$. This follows again from \cite{kalliongis1991symmetries}.

\begin{lem} The quotient of a solid torus under a $\mathbb{Z}_{m}$-action with generator $\varphi(g_{1})(u,v)=(e^{2a\pi i}u,e^{2b\pi i}v)$ is a solid torus with exceptional core of order:

$$k=\frac{b_{2}}{gcd(a_{2},b_{2})}$$

\end{lem}

\begin{proof} So $\varphi(g_{1})^{a_{2}}(u,v)=(u,e^{2a_{2}b\pi i}v)$ and $\varphi(g_{1})^{a_{2}}$ is an order $\frac{lcm(a_{2},b_{2})}{a_{2}}=\frac{b_{2}}{gcd(a_{2},b_{2})}$ element that fixes the core. The quotient space then has an exceptional core of order $k=\frac{b_{2}}{gcd(a_{2},b_{2})}$.
\end{proof}

\begin{lem} The quotient of a solid torus under a $\mathbb{Z}_{m}\times\mathbb{Z}_{l}$-action where $m$ divides $l$ with generators $\varphi(g_{1})(u,v)=(e^{2a\pi i}u,e^{2b\pi i}v)$ and $\varphi(g_{2})(u,v)=(e^{2c\pi i}u,e^{2d\pi i}v)$ is a solid torus with exceptional core of order:

$$k=\frac{b_{2}d_{2}gcd(a_{2},c_{2})}{gcd(d_{2}gcd(a_{2},c_{2})gcd(a_{2},b_{2}),a_{2}b_{1}c_{1}d_{2}z+b_{2}c_{2}d_{1})}$$

Where $z$ is such that $\frac{a_{1}z+1}{a_{2}}\in\mathbb{Z}$.
\end{lem}

\begin{proof} We begin by noting that the quotient of the solid torus under the normal group generated by $g_{1}$ is $V(k')$ where $k'=\frac{b_{2}}{gcd(a_{2},b_{2})}$. We then claim the projection under the restricted action of $\left\langle g_{1}\right\rangle$, is $p_{g_{1}}(u,v)=(u^{a_{2}},u^{z'}v^{k'})$ where $z'=\frac{a_{2}b_{1}}{gcd(a_{2},b_{2})}z$ for $z$ such that $a_{1}z+a_{2}y=-1$. 

To prove this, we first note that:

$$p_{g_{1}}(e^{2a\pi i}u,e^{2b\pi i}v)=((e^{2a\pi i}u)^{a_{2}},(e^{2a\pi i}u)^{z'}(e^{2b\pi i}v)^{k'})=(u^{a_{2}},e^{2\pi i(az'+bk')}u^{z'}v^{k'})$$

So then:

$$az'+bk'=\frac{b_{1}}{gcd(a_{2},b_{2})}a_{1}z+\frac{b_{1}}{gcd(a_{2},b_{2})}=\frac{a_{1}b_{1}}{gcd(a_{2},b_{2})}(-1-a_{2}y)+\frac{b_{1}}{gcd(a_{2},b_{2})}=\frac{-a_{2}y}{gcd(a_{2},b_{2})}\in\mathbb{Z}$$

So that:

$$p_{g_{1}}(e^{2a\pi i}u,e^{2b\pi i}v)=(u^{a_{2}},u^{z'}v^{k'})$$

Also, if we solve $p_{g_{1}}(u,v)=(1,1)$, we yield $u^{a_{2}}=1$ and $u^{z'}v^{k'}=1$. So that there are $a_{2}k'=\frac{a_{2}b_{2}}{gcd(a_{2},b_{2})}$ possible solutions. This is the order of $g_{1}$.

Now, there is an induced map $\overline{\varphi(g_{2})}$ such that $\overline{\varphi(g_{2})}\circ p_{g_{1}}=p_{g_{1}}\circ\varphi(g_{2})$. 

We compute:

$$\overline{\varphi(g_{2})}(u^{a_{2}},u^{z'}v^{k'})	=(\overline{\varphi(g_{2})}\circ p_{g_{1}})(u,v)
	=(p_{g_{1}}\circ\varphi(g_{2}))(u,v)
	=p_{g_{1}}(e^{2c\pi i}u,e^{2d\pi i}v)
	=(e^{2a_{2}c\pi i}u^{a_{2}},e^{2\pi i(cz'+dk')}u^{z'}v^{k'})$$
	
It follows that $\overline{\varphi(g_{2})}(u,v)=(e^{2\pi a_{2}ci}u,e^{2\pi(cz'+dk')i}v)$.

So then $\overline{\varphi(g_{2})}^{\frac{c_{2}}{gcd(a_{2},c_{2})}}(u,v)=(u,e^{2\pi\frac{c_{2}}{gcd(a_{2},c_{2})}(cz'+dk')i}v)$.

Now $\overline{\varphi(g_{2})}^{\frac{c_{2}}{gcd(a_{2},c_{2})}}$ is an element that fixes the core of $V(k')$ and is of order the denominator of $\frac{c_{2}}{gcd(a_{2},c_{2})}(cz'+dk')$ when in reduced form. We calculate:

$$\frac{c_{2}}{gcd(a_{2},c_{2})}(cz'+dk')=\frac{c_{1}d_{2}z'+d_{1}c_{2}k'}{d_{2}gcd(a_{2},c_{2})}$$

Hence $\overline{\varphi(g_{2})}^{\frac{c_{2}}{gcd(a_{2},c_{2})}}$ has order $\frac{d_{2}gcd(a_{2},c_{2})}{gcd(d_{2}gcd(a_{2},c_{2}),c_{1}d_{2}z'+c_{2}d_{1}k')}$.

So finally, the order of the exceptional core of quotient space of the whole action is:

\begin{align*}
k	&=k'\frac{d_{2}gcd(a_{2},c_{2})}{gcd(d_{2}gcd(a_{2},c_{2}),c_{1}d_{2}z'+c_{2}d_{1}k')} \\
	&=\frac{b_{2}}{gcd(a_{2},b_{2})}\frac{d_{2}gcd(a_{2},c_{2})}{gcd(d_{2}gcd(a_{2},c_{2}),c_{1}d_{2}z'+c_{2}d_{1}k')} \\
	&=\frac{b_{2}d_{2}gcd(a_{2},c_{2})}{gcd(d_{2}gcd(a_{2},b_{2})gcd(a_{2},c_{2}),c_{1}d_{2}a_{2}b_{1}z+c_{2}d_{1}b_{2})} \\
\end{align*}

\end{proof}

We now consider an action of the stabilizer that reverses the orientation of the fibers. By \cite{kalliongis1991symmetries} the action will be a $Dih(\mathbb{Z}_{m}\times\mathbb{Z}_{l})$-action where $m$ divides $l$ with generators $\varphi(g_{1})(u,v)=(e^{2a\pi i}u,e^{2b\pi i}v)$, $\varphi(g_{2})(u,v)=(e^{2c\pi i}u,e^{2d\pi i}v)$, and $\varphi(g_{3})(u,v)=(u^{-1},v^{-1})$. We here note that similar to the proof of Lemma 6.2. we can consider the quotient of the $\mathbb{Z}_{m}\times\mathbb{Z}_{l}$-action and then the induced involution upon it. The following lemma then holds:

\begin{lem} The quotient of a solid torus under a $Dih(\mathbb{Z}_{m}\times\mathbb{Z}_{l})$-action where $m$ divides $l$ with generators $\varphi(g_{1})(u,v)=(e^{2a\pi i}u,e^{2b\pi i}v)$, $\varphi(g_{2})(u,v)=(e^{2c\pi i}u,e^{2d\pi i}v)$, and $\varphi(g_{3})(u,v)=(u^{-1},v^{-1})$ is a Conway ball with exceptional core of order:

$$k=\frac{b_{2}d_{2}gcd(a_{2},c_{2})}{gcd(d_{2}gcd(a_{2},c_{2})gcd(a_{2},b_{2}),a_{2}b_{1}c_{1}d_{2}z+b_{2}c_{2}d_{1})}$$

Where $z$ is such that $\frac{a_{1}z+1}{a_{2}}\in\mathbb{Z}$.
\end{lem}

\begin{proof}
This follows from considering an orientation-preserving involution on $V(k)$ the quotient of the $\mathbb{Z}_{m}\times\mathbb{Z}_{l}$-action generated by $\varphi(g_{1})(u,v)=(e^{2a\pi i}u,e^{2b\pi i}v)$ and $\varphi(g_{2})(u,v)=(e^{2c\pi i}u,e^{2d\pi i}v)$.
\end{proof}

\subsection{$d'|_{\partial V_{i}/Stab_{\varphi_{X}}(\partial V_{i})}$}

We again begin by assuming that the action preserves the orientation of the fibers. So now $\hat{M}/\hat{\varphi}$ has a collection of boundary tori. These will be filled by solid tori with a possible exceptional core. It remains to show how the gluing map from the boundary of the solid tori into $\hat{M}/\hat{\varphi}$ will look.

By using product structures $k:S^{1}\times F\rightarrow\hat{M}$ and $k':S^{1}\times F/\hat{\varphi}_{F}\rightarrow\hat{M}/\hat{\varphi}$ that restrict to positively oriented product structures $k_{T_{i}}:S^{1}\times S^{1}\rightarrow T_{i}$ and $k'_{T'_{i}}:S^{1}\times S^{1}\rightarrow T'_{i}$, we can consider:

$$H_{1}(\hat{M})=\left\langle t,x_{1},\ldots,x_{s},a_{1},b_{1},\ldots,a_{g},b_{g}|x_{1}\cdots x_{s}=1,\textrm{all commute}\right\rangle$$

$$H_{1}(\hat{M}/\hat{\varphi})=\left\langle t',x'_{1},\ldots,x'_{s'},a'_{1},b'_{1},\ldots,a'_{g'},b'_{g'}|x'_{1}\cdots x'_{s'}=1,\textrm{all commute}\right\rangle $$

Here note that we again allow $F$ to be more generally any orientable surface with boundary as it presents no extra complication to the calculations.
$t,t'$ represent a fiber of $\hat{M}$ and $\hat{M}/\hat{\varphi}$ respectively; $x_{1},\ldots,x_{s}$ represent the boundary loops of $k(\{1\}\times F)$; and similarly $x'_{1},\ldots,x'_{s'}$ represent the boundary loops of $k'(\{1\}\times F/\hat{\varphi}_{F})$.

We then we have that:

$$(p_{\hat{\varphi}})_{*}(t)=t'^{a}, (p_{\hat{\varphi}})_{*}(x_{i})=x_{j(i)}^{\prime m_{j(i)}}t^{\prime l_{j(i)}}$$

Here $j:\{1,\ldots,s\}\rightarrow\{1,\ldots,s'\}$ is a surjection.

Note that this is well-defined as if $j(i_{1})=j(i_{2})$ then $(p_{\hat{\varphi}})_{*}(x_{i_{1}})=(p_{\hat{\varphi}})_{*}(\varphi(g)_{*}(x_{i_{2}}))$ for some $g\in G$. 

Now, $1=(p_{\hat{\varphi}})_{*}(x_{1}\cdots x_{s})=x_{1}^{\prime m_{1}\#j^{-1}(1)}\cdots x_{s'}^{\prime m_{s'}\#j^{-1}(s')}t^{\prime l_{1}\#j^{-1}(1)+\ldots+l_{s'}\#j^{-1}(s')}$.

So that $0=l_{1}\#j^{-1}(1)+\ldots+l_{s'}\#j^{-1}(s')$ and necessarily $m_{j(i_{1})}\#j^{-1}(i_{1})=m_{j(i_{2})}\#j^{-1}(i_{2})$ for any $i_{1},i_{2}\in\{1,\ldots,s\}$.

Now, for any torus (either on the boundary of $\hat{M}$ or on the boundary of one of the solid tori) we have that $Stab(T)\cong\mathbb{Z}_{m}$ or $\mathbb{Z}_{m}\times\mathbb{Z}_{l}$ where $m$ divides $l$. \cite{kalliongis1991symmetries}

So we consider the diagram:

$$\begin{array}{ccccc}
 &  & d|_{\partial V_{i}}\\
 & T_{i} & \leftarrow & \partial V_{i}\\
p_{\hat{\varphi}}|_{T_{i}} & \downarrow &  & \downarrow & p_{\varphi_{X}}|_{\partial V_{i}}\\
 & T'_{i'}=T_{i}/Stab(T_{i}) & \leftarrow & \partial V'_{i'}\\
 &  & d'|_{\partial V'_{i'}}
\end{array}$$

We begin with the cyclic case. We can then choose the product structure such that $(k_{T_{i}}^{-1}\circ\hat{\varphi}(g)\circ k_{T_{i}})(u,v)=(e^{2a\pi i}u,e^{2b\pi i}v)$ for some $a,b\in\mathbb{Q}$ and $g$ a generator of $Stab(T_{i})$. Letting $a=\frac{a_{1}}{a_{2}},b=\frac{b_{1}}{b_{2}}$ be fully reduced, this has order $lcm(a_{2},b_{2})$.

\begin{lem} If $Stab(T_{i})\cong\mathbb{Z}_{m}$, then $(k_{T'_{i'}}^{-1}\circ p_{\hat{\varphi}}\circ k_{T_{i}})(u,v)=(u^{\frac{lcm(a_{2},b_{2})}{b_{2}}}v^{l_{j(i)}},v^{b_{2}})$ where $l_{j(i)}$ is an integer such that $\frac{a_{1}lcm(a_{2},b_{2})}{a_{2}b_{2}}+l_{j(i)}\frac{b_{1}}{b_{2}}=\frac{a_{1}}{gcd(a_{2},b_{2})}+l_{j(i)}\frac{b_{1}}{b_{2}}$ is integer valued.
\end{lem}

\begin{proof} The projection $p_{\hat{\varphi}}|_{T_{i}}$ will need to send a fiber to a fiber, hence $(k_{T'_{i'}}^{-1}\circ p_{\hat{\varphi}}\circ k_{T_{i}})(u,v)=(u^{r}v^{s},v^{t})$. But now:

$$(u^{r}v^{s},v^{t})=(k_{T'_{i'}}^{-1}\circ p_{\hat{\varphi}}\circ k_{T_{i}})(e^{2a\pi i}u,e^{2b\pi i}v))=(e^{2(ar+bs)\pi i}u^{r}v^{s},e^{2bt\pi i}v^{t})$$

So then take $t=b_{2}$.

Now consider $(k_{T'_{i'}}^{-1}\circ p_{\hat{\varphi}}\circ k_{T_{i}})(u,v)=(u^{r}v^{s},v^{t})=(1,1)$. This should have $lcm(a_{2},b_{2})$ solutions. So $u^{r}=1$ has $\frac{lcm(a_{2},b_{2})}{b_{2}}$ solutions and $r=\frac{lcm(a_{2},b_{2})}{b_{2}}$.

Now, $ar+bs=a\frac{lcm(a_{2},b_{2})}{b_{2}}+bs\in\mathbb{Z}$. So let $s=l_{j(i)}$ be a solution to this. This exists as $gcd(a_{2},b_{2})$ divides $b_{2}$ by \cite{olds2000geometry}. There are however an infinite number of choices depending upon the product structure $k'_{T'_{i}}:S^{1}\times S^{1}\rightarrow T'_{i}$.

\end{proof}

The projection $p_{\varphi_{X}}|_{\partial V_{i}}$ will need to extend over the entire solid torus and so will need to send a meridian to a meridian. We can again choose the product structure such that $(k_{\partial V_{i}}^{-1}\circ\varphi_{X}(g)\circ k_{\partial V_{i}})(u,v)=(e^{2a\pi i}u,e^{2b\pi i}v)$. Hence $p_{\varphi_{X}}|_{\partial V_{i}}$ will similarly give $(k_{\partial V'_{i'}}^{-1}\circ p_{\hat{\varphi}}\circ k_{\partial V_{i}})(u,v)(u^{a_{2}},u^{z}v^{\frac{lcm(a_{2},b_{2})}{a_{2}}})$. Here the choice of $z$ will not affect the filling but depends upon the product structure $k'_{\partial V'_{i'}}:S^{1}\times S^{1}\rightarrow\partial V'_{i'}$.

We proceed with $Stab(T_{i})\cong\mathbb{Z}_{m}\times\mathbb{Z}_{l}$ where $m$ divides $l$. Then $(k_{T_{i}}^{-1}\circ\hat{\varphi}(g_{1})\circ k_{T_{i}})(u,v)=(e^{2a\pi i}u,e^{2b\pi i}v)$ and $(k_{T_{i}}^{-1}\circ\hat{\varphi}(g_{1})\circ k_{T_{i}})(u,v)=(e^{2c\pi i}u,e^{2d\pi i}v)$ for some $a,b,c,d\in\mathbb{Q}$ and $g_{1},g_{2}$ generators of $Stab(T_{i})$. Here $m=lcm(a_{2},b_{2})$ and $l=lcm(c_{2},d_{2})$. 

\begin{lem} If $Stab(T_{i})\cong\mathbb{Z}_{m}\times\mathbb{Z}_{l}$, then $(k_{T'_{i'}}^{-1}\circ p_{\hat{\varphi}}\circ k_{T_{i}})(u,v)=(u^{\frac{ml}{lcm(b_{2},d_{2})}}v^{l_{j(i)}},v^{lcm(b_{2},d_{2})})$ where $lcm(b_{2},d_{2})$ divides $l_{j(i)}$.
\end{lem}

\begin{proof}
The projection $p_{\hat{\varphi}}|_{T_{i}}$ will again need to send a fiber to a fiber. Hence it will again be of the form $(k_{T'_{i'}}^{-1}\circ p_{\hat{\varphi}}\circ k_{T_{i}})(u,v)=(u^{r}v^{s},v^{t})$. Now:

$$(u^{r}v^{s},v^{t})=(k_{T'_{i'}}^{-1}\circ p_{\hat{\varphi}}\circ k_{T_{i}})(e^{2a\pi i}u,e^{2b\pi i}v))=(e^{2(ar+bs)\pi i}u^{r}v^{s},e^{2bt\pi i}v^{t})$$

$$(u^{r}v^{s},v^{t})=(k_{T'_{i'}}^{-1}\circ p_{\hat{\varphi}}\circ k_{T_{i}})(e^{2c\pi i}u,e^{2d\pi i}v))=(e^{2(cr+ds)\pi i}u^{r}v^{s},e^{2dt\pi i}v^{t})$$

Hence we take $t=lcm(b_{2},d_{2})$.

Now consider $(k_{T'_{i'}}^{-1}\circ p_{\hat{\varphi}}\circ k_{T_{i}})(u,v)=(u^{r}v^{s},v^{t})=(1,1)$. This should have $ml$ solutions. So $u^{r}=1$ has $\frac{ml}{lcm(b_{2},d_{2})}$ solutions and $r=\frac{ml}{lcm(b_{2},d_{2})}$.

Finally, $s$ is such that $\frac{aml}{lcm(b_{2},d_{2})}+bs\in\mathbb{Z}$ and $\frac{cml}{lcm(b_{2},d_{2})}+ds\in\mathbb{Z}$. 

We now calculate:
\begin{align*}
    lcm(a_{2},b_{2})lcm(c_{2},d_{2})	&=lcm(lcm(a_{2},b_{2}),lcm(c_{2},d_{2}))gcd(lcm(a_{2},b_{2}),lcm(c_{2},d_{2}))\\
	&=lcm(lcm(a_{2},b_{2}),lcm(c_{2},d_{2}))lcm(a_{2},b_{2})\\
\end{align*}

	So that:
\begin{align*}
	lcm(c_{2},d_{2})	&=lcm(lcm(a_{2},b_{2}),lcm(c_{2},d_{2}))\\
	&=lcm(lcm(a_{2},c_{2}),lcm(b_{2},d_{2}))
\end{align*} 
Then:

$$\frac{aml}{lcm(b_{2},d_{2})}=\frac{a_{1}lcm(a_{2},b_{2})lcm(c_{2},d_{2})}{a_{2}lcm(b_{2},d_{2})}=\frac{a_{1}lcm(a_{2},b_{2})lcm(lcm(a_{2},c_{2}),lcm(b_{2},d_{2}))}{a_{2}lcm(b_{2},d_{2})}\in\mathbb{Z}$$

Similarly, $\frac{cml}{lcm(b_{2},d_{2})}\in\mathbb{Z}$. So then we require that $bs,ds\in\mathbb{Z}$. Hence, $b_{2}$ and $d_{2}$ must divide $s$ and we take $s=l_{j(i)}$ to be a multiple of $lcm(b_{2},d_{2})$. \end{proof}

The projection $p_{\varphi_{X}}|_{\partial V_{i}}$ will need to extend over the entire solid torus and so will need to send a meridian to a meridian. We can again choose the product structure such that $(k_{\partial V_{i}}^{-1}\circ\varphi_{X}(g_{1})\circ k_{\partial V_{i}})(u,v)=(e^{2a\pi i}u,e^{2b\pi i}v)$ and $(k_{\partial V_{i}}^{-1}\circ\varphi_{X}(g_{1})\circ k_{\partial V_{i}})(u,v)=(e^{2c\pi i}u,e^{2d\pi i}v)$. Hence $p_{\varphi_{X}}|_{\partial V_{i}}$ will similarly give $(k_{\partial V'_{i'}}^{-1}\circ p_{\hat{\varphi}}\circ k_{\partial V_{i}})(u,v)=(u^{lcm(a_{2},c_{2})},u^{z}v^{\frac{ml}{lcm(a_{2},c_{2})}})$. Here the choice of $z$ will not again affect the filling but depends upon the product structure $k'_{\partial V'_{i'}}:S^{1}\times S^{1}\rightarrow\partial V'_{i'}$.

So now we have from above that $0=l_{1}\#j^{-1}(1)+\ldots+l_{s'}\#j^{-1}(s')$. Hence we have the degree of freedom to choose $c_{1},\ldots,c_{s'-1}$ (according to the conditions), but then $c_{s'}$ will be uniquely determined.

Each filling $d'|_{\partial V'_{i'}}$ will now be determined be solving:

$$(p_{\hat{\varphi}}|_{T_{i}})_{*}(d|_{\partial V_{i}})_{*}=(d'|_{\partial V'_{i'}})_{*}(p_{\varphi_{X}}|_{\partial V_{i}})_{*}$$

\begin{exmp} 

We consider a $Dih(\mathbb{Z}_{6}\times\mathbb{Z}_{12})$-action on the lens space $M=(0,o_{1}|(3,2),(1,5))$ constructed by:

$$f_{1}:S^{1}\times A\rightarrow S^{1}\times A, f_{1}(u,\rho v)=(e^{\frac{2\pi i}{6}}u,\rho e^{\frac{2\pi i}{3}}v)$$

$$f_{2}:S^{1}\times A\rightarrow S^{1}\times A, f_{1}(u,\rho v)=(u,\rho e^{\frac{2\pi i}{12}}v)$$

$$f_{2}:S^{1}\times A\rightarrow S^{1}\times A, f_{1}(u,\rho v)=(u^{-1},\rho v^{-1})$$

Here we parameterize $A=\{\rho v|1\leq\rho\leq2,v\in S^{1}\}$. Note that according to the product structure $S^{1}\times A$ one boundary torus is positively oriented and the other negatively depending on the orientation on the fiber. We take $S^{1}\times S^{1}$ to be positively oriented and $S^{1}\times2S^{1}$ to be negatively oriented.

We calculate first $(S^{1}\times A)/Dih(\mathbb{Z}_{6}\times\mathbb{Z}_{12})$. This will be $I\times A$ where $(0,\rho v)$ is identified with $(0,\rho v^{-1})$ and $(0,\rho v)$ is identified with $(0,\rho v^{-1})$ with four arcs of order 2. It will be $S^{2}\times S^{1}$ with four properly embedded arcs looking as shown in Figure 3: 

\begin{figure}[ht]
\centering
\includegraphics[height=4cm]{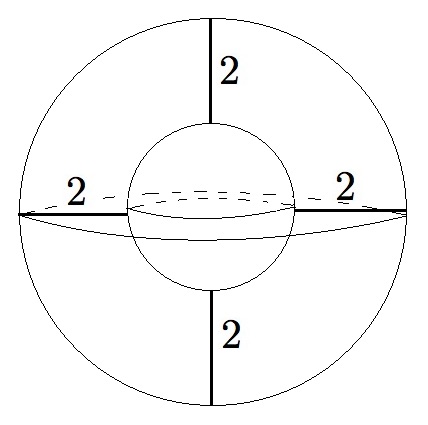}
\caption{Quotient space $(S^{1}\times A)/Dih(\mathbb{Z}_{6}\times\mathbb{Z}_{12})$}
\end{figure}

Next we compute the orders of the exceptional sets of the two Conway balls that fill the two boundary components. We first calculate the generators of the induced action on the solid tori $V_{1}$ and $V_{2}$ that correspond to the fillings $(3,2)$ and $(1,5)$.

Firstly, for $V_{1}$, we compute:

\begin{align*}
(d|_{\partial V_{1}}^{-1}\circ f_{1}\circ d|_{\partial V_{1}})(u,v)	&=(d|_{\partial V_{1}}^{-1}\circ f_{1})(u^{-1}v^{2},u^{-1}v^{3})\\
	&=d|_{\partial V_{1}}^{-1}(e^{\frac{2\pi i}{6}}u^{-1}v^{2},e^{\frac{2\pi i}{3}}u^{-1}v^{3})\\
	&=(e^{2\pi i(\frac{-3}{6}+\frac{2}{3})}u,e^{2\pi i(\frac{-1}{6}+\frac{1}{3})}v)\\
	&=(e^{\frac{2\pi i}{6}}u,e^{\frac{2\pi i}{6}}v)(d|_{\partial V_{1}}^{-1}\circ f_{2}\circ d|_{\partial V_{1}})(u,v)\\	&=(d|_{\partial V_{1}}^{-1}\circ f_{2})(u^{-1}v^{2},u^{-1}v^{3})\\
	&=d|_{\partial V_{1}}^{-1}(u^{-1}v^{2},e^{\frac{2\pi i}{12}}u^{-1}v^{3})\\
	&=(e^{2\pi i(\frac{2}{12})}u,e^{2\pi i(\frac{3}{12})}v)\\
	&=(e^{\frac{2\pi i}{6}}u,e^{\frac{2\pi i}{4}}v)
\end{align*}

	So then by Lemma 6.6, the exceptional set will have order:

$$k=\frac{(6)(12)gcd(6,6)}{gcd(12gcd(6,6)gcd(6,6),(6)(1)(1)(12)z+(6)(6)(11))}=\frac{432}{gcd(432,72z+396)}$$

Here $z$ is such that $\frac{z+1}{6}\in\mathbb{Z}$. So take $z=-1$ and then $k=\frac{432}{gcd(432,324)}=\frac{432}{108}=4$.

Secondly, for $V_{2}$, we compute:
\begin{align*}
(d|_{\partial V_{2}}^{-1}\circ f_{1}\circ d|_{\partial V_{2}})(u,v)	&=(d|_{\partial V_{2}}^{-1}\circ f_{1})(u^{-1}v^{5},v)\\
	&=d|_{\partial V_{2}}^{-1}(e^{\frac{2\pi i}{6}}u^{-1}v^{5},e^{-\frac{2\pi i}{3}}v)\\
	&=(e^{2\pi i(\frac{-1}{6}-\frac{5}{3})}u,e^{-\frac{2\pi i}{3}}v)\\
	&=(e^{\frac{2\pi i}{6}}u,e^{\frac{4\pi i}{3}}v)(d|_{\partial V_{2}}^{-1}\circ f_{2}\circ d|_{\partial V_{2}})(u,v)\\	&=(d|_{\partial V_{2}}^{-1}\circ f_{2})(u^{-1}v^{5},v)\\
	&=d|_{\partial V_{2}}^{-1}(u^{-1}v^{5},e^{-\frac{2\pi i}{12}}v)\\
	&=(e^{2\pi i(\frac{-5}{12})}u,e^{-\frac{2\pi i}{12}}v)\\
	&=(e^{\frac{14\pi i}{12}}u,e^{\frac{10\pi i}{12}}v)\\
\end{align*}

	So then again using Lemma 6.6, the exceptional set will have order:

$$k=\frac{(3)(12)gcd(2,12)}{gcd(12gcd(2,12)gcd(2,3),(2)(2)(7)(12)z+(3)(12)(5))}=\frac{72}{gcd(24,288z+180)}$$

Here $z$ is such that $\frac{z+1}{2}\in\mathbb{Z}$. So take $z=-1$ and then:

$$k=\frac{72}{gcd(24,108)}=\frac{72}{12}=6$$

We now compute the projection maps. By section 6.1, the projection map from both $S^{1}\times S^{1}$ and $S^{1}\times2S^{1}$ will have the matrix:

$$\left[\begin{array}{cc}
\frac{ml}{lcm(b_{2},d_{2})} & c\\
0 & lcm(b_{2},d_{2})
\end{array}\right]=\left[\begin{array}{cc}
6 & c\\
0 & 12
\end{array}\right]$$

Here $lcm(b_{2},d_{2})=12$ divides $c$, so we take $c=12$.

The projection map from $\partial V_{1}$ will have matrix: 

$$\left[\begin{array}{cc}
lcm(a_{2},c_{2}) & 0\\
z & \frac{ml}{lcm(a_{2},c_{2})}
\end{array}\right]=\left[\begin{array}{cc}
6 & 0\\
6 & 12
\end{array}\right]$$

The projection map from $\partial V_{2}$ will have matrix: 

$$\left[\begin{array}{cc}
lcm(a_{2},c_{2}) & 0\\
z & \frac{ml}{lcm(a_{2},c_{2})}
\end{array}\right]=\left[\begin{array}{cc}
12 & 0\\
6 & 6
\end{array}\right]$$

We now calculate the projected filling of $S^{1}\times S^{1}$ with $V_{1}$ by solving:

$$(p_{\hat{\varphi}}|_{S^{1}\times S^{1}})_{*}(d|_{\partial V_{1}})_{*}=(d'|_{\partial V'_{1}})_{*}(p_{\varphi_{X}}|_{\partial V_{1}})_{*}$$

$$\left[\begin{array}{cc}
6 & 12\\
0 & 12
\end{array}\right]\left[\begin{array}{cc}
-1 & 2\\
-1 & 3
\end{array}\right]=\left[\begin{array}{cc}
x' & p'\\
y' & q'
\end{array}\right]\left[\begin{array}{cc}
6 & 0\\
6 & 12
\end{array}\right]$$

This yields:

$$\left[\begin{array}{cc}
-18 & 48\\
-12 & 36
\end{array}\right]=\left[\begin{array}{cc}
6x'+6p' & 12p'\\
6y'+6q' & 12q'
\end{array}\right]$$

So then $p'=4, q'=3, x'=-7,$ and $y'=-5$.

We now calculate the projected filling of $S^{1}\times2S^{1}$ with $V_{2}$ by solving:

$(p_{\hat{\varphi}}|_{S^{1}\times2S^{1}})_{*}(d|_{\partial V_{2}})_{*}=(d'|_{\partial V'_{2}})_{*}(p_{\varphi_{X}}|_{\partial V_{2}})_{*}$

$$\left[\begin{array}{cc}
6 & 12\\
0 & 12
\end{array}\right]\left[\begin{array}{cc}
-1 & 5\\
0 & 1
\end{array}\right]=\left[\begin{array}{cc}
x' & p'\\
y' & q'
\end{array}\right]\left[\begin{array}{cc}
12 & 0\\
6 & 6
\end{array}\right]$$

This yields:

$$\left[\begin{array}{cc}
-6 & 42\\
0 & 12
\end{array}\right]=\left[\begin{array}{cc}
12x'+6p' & 6p'\\
12y'+6q' & 6q'
\end{array}\right]$$

So then $p'=7,q'=2,  x'=-4,$ and $y'=-1$. 

This fully characterizes the quotient space. We visualize in Figure 4:

\begin{figure}[ht]
\centering
\includegraphics[height=4cm]{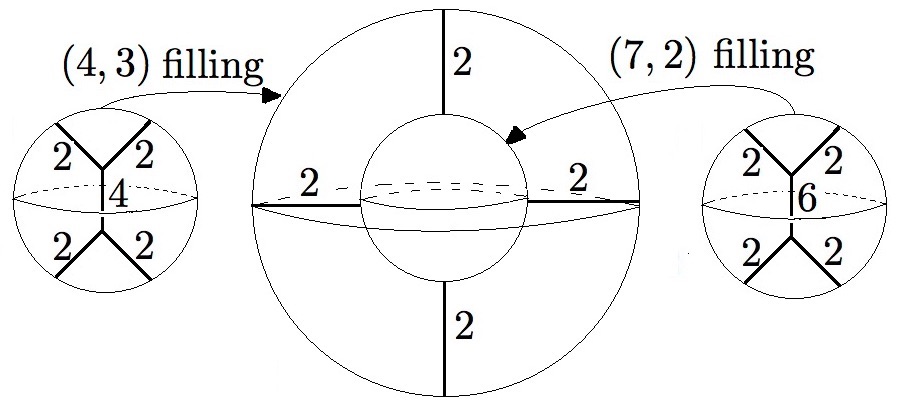}
\caption{Full quotient space}
\end{figure}

\end{exmp}

\section{Summary of results and future work}
In this paper we have studied the group actions on Seifert fibered elliptic manifolds using the results of \cite{peet2018} and \cite{peet2018finite}. We have extended the results of those papers by considering when an orientation-reversing action is possible and shown this can only happen if there are no critical fibers of order greater than $2$ and the Euler class is non-zero. These results allowed us to consider the possible base spaces of the Seifert manifolds and determine what the possible group actions are. As future work, Seifert manifolds that do admit orientation-reversing actions could be considered as well as a construction of such an action.

\bibliographystyle{unsrt}
\bibliography{references}

\begin{thebibliography}{10}

\bibitem{peet2018}
Benjamin Peet.
\newblock Finite, fiber- and orientation-preserving group actions on totally
  orientable seifert manifolds.
\newblock {\em Annales Mathematicae Silesianae}, (0), 2019.

\bibitem{peet2018finite}
Benjamin Peet.
\newblock Finite, fiber-and orientation-preserving actions on orientable
  seifert manifolds with non-orientable base space.
\newblock {\em JP Journal of Geometry and Topology}, 23, 2019.

\bibitem{lee2003smooth}
J.~M. Lee.
\newblock Smooth manifolds.
\newblock In {\em Introduction to Smooth Manifolds}, pages 1--29. Springer,
  2003.

\bibitem{neumann1978seifert}
W.~D. Neumann and F.~Raymond.
\newblock Seifert manifolds, plumbing, $\mu$-invariant and orientation
  reversing maps.
\newblock In {\em Algebraic and geometric topology}, pages 163--196. Springer,
  1978.

\bibitem{thurstongeometry}
W.~P. Thurston.
\newblock The geometry and topology of 3-manifolds.
\newblock {\em Lecture Notes}, 1979.

\bibitem{kalliongis2018}
John Kalliongis and Ryo Ohashi.
\newblock Finite actions on the 2-sphere, the projective plane and i-bundles
  over the projective plane.
\newblock {\em ARS MATHEMATICA CONTEMPORANEA}, 15(2):297--321, 2018.

\bibitem{kalliongis2002geometric}
J.~Kalliongis and A.~Miller.
\newblock Geometric group actions on lens spaces.
\newblock {\em Kyungpook mathematical journal}, 42(2):313--313, 2002.

\bibitem{scott1983geometries}
P.~Scott.
\newblock The geometries of 3-manifolds.
\newblock {\em Bulletin of the London Mathematical Society}, 15(5):401--487,
  1983.

\bibitem{kalliongis1991symmetries}
J.~Kalliongis and A.~Miller.
\newblock The symmetries of genus one handlebodies.
\newblock {\em Canad. J. Math}, 43(19911):371--404, 1991.

\bibitem{olds2000geometry}
G.~Davidoff C.~D.~Olds, A.~Lax and G.~P. Davidoff.
\newblock {\em The geometry of numbers}, volume~41.
\newblock Cambridge University Press, 2000.

\end{thebibliography}

\end{document}